\documentclass[final,leqno]{siamltex}

\usepackage{amsmath,amssymb,amscd,amsxtra,amsfonts,mathrsfs}
\usepackage{epsf,graphicx,epsfig,color,latexsym,cite}

\title{Inverse Random Source Scattering for Elastic Waves}

\author{Gang Bao \thanks{School of Mathematical Sciences, Zhejiang University,
Hangzhou 310027, China; The research was supported in
part by a Key Project of the Major Research Plan of NSFC (No. 91130004),
an NSFC A3 Project (No.11421110002), NSFC Tianyuan Projects (No. 11426235; No.
11526211), and a special research grant from Zhejiang University. ({\tt
baog@zju.edu.cn})} \and Chuchu Chen \thanks{Department of Mathematics, Michigan
State University, East Lansing, MI 48824, USA. ({\tt
chen2095@math.purdue.edu})} \and Peijun Li \thanks{Department of Mathematics,
Purdue University, West Lafayette, Indiana 47907, USA. This author's research
was supported in part by the NSF grant DMS-1151308. ({\tt
lipeijun@math.purdue.edu})}}

\begin{document}

\maketitle

\begin{abstract}
This paper is concerned with the direct and inverse random source scattering
problems for elastic waves where the source is assumed to be driven by an
additive white noise. Given the source, the direct problem is to determine the
displacement of the random wave field. The inverse problem is to reconstruct the
mean and variance of the random source from the boundary measurement of the wave
field at multiple frequencies. The direct problem is shown to have a unique mild
solution by using a constructive proof. Based on the explicit mild solution,
Fredholm integral equations of the first kind are deduced for the inverse
problem. The regularized Kaczmarz method is presented to solve the ill-posed
integral equations. Numerical experiments are included to demonstrate the
effectiveness of the proposed method.
\end{abstract}



\pagestyle{myheadings}
\thispagestyle{plain}
\markboth{G. Bao, C. Chen, and P. Li}{Inverse Random Source Scattering for
Elastic Waves}

\section{Introduction}

The inverse source scattering problems, an important research subject in inverse
scattering theory, are to determine the unknown sources that generate prescribed
radiated wave pattens \cite{CK-98, I-89}. These problems are largely motivated
by applications in medical imaging \cite{FKM-IP04}. A typical example is to use
electric or magnetic measurements on the surface of the human body, such as
head, to infer the source currents inside the body, such as the brain, that
produce the measured data. Mathematically, the inverse source scattering
problems have been widely examined for acoustic and electromagnetic waves by
many researchers \cite{AM-IP06, ABF-SJAM02, BN-IP11, BLT-JDE10, BLRX-SJNA15,
ZG-IP15, DML-SJAM07, EV-IP09, MD-IEEE99}. For instance, it is known that the
inverse source problem does not have a unique solution at a fixed frequency due 
to the existence of non-radiating sources \cite{DS-IEEE82, HKP-IP05}; it is
ill-posed as small variations in the measured data can lead to huge errors in
the reconstructions \cite{BLLT-IP15}. 

Although the deterministic counterparts have been well studied, little is known
for the stochastic inverse problems due to uncertainties, which are widely
introduced to the models for two common reasons: randomness may directly appear
in the studied systems \cite{E-13, F-06} and incomplete knowledge of the systems
must be modeled by uncertainties \cite{KE-05}. A uniqueness result may be found
in \cite{D-JMP79}, where it showed that the auto-correlation function of the
random source was uniquely determined everywhere outside the source region by
the auto-correlation function of the radiated field. Recently, one-dimensional
stochastic inverse source problems have been considered in \cite{BCLZ-MC14,
BX-IP13, L-IP11}, where the governing equations are stochastic ordinary
differential equations. Utilizing the Green functions, the authors have
presented the first approach in \cite{BCL} for solving the inverse random
source scattering problem in higher dimensions, where the stochastic
partial differential equations are considered. 

In this paper, we study both the direct and inverse random source scattering
problems for elastic waves. Given the source, the direct problem is to determine
the displacement of the random wave field. The inverse problem is to reconstruct
the mean and variance of the random source from the boundary measurement of the
wave field. Recently, the elastic wave scattering problems have received ever
increasing attention for their significant applications in many scientific areas
such as geophysics and seismology \cite{BC-IP05, LL-86, LWZ-IP15}. For example,
they have played an important role in the problem for elastic pulse transmission
and reflection through the Earth when investigating earthquakes and determining
their focus, which is exactly the motivation of this work. 

The random source is assumed to be driven by a white noise, which can be thought
as the derivative of a Brownian sheet or a multi-parameter Brownian motion. The
goal is to determine the mean and variance of the random source function by
using the same statistics of the displacement of the wave field, which are
measured on a boundary enclosing the compactly supported source at multiple
frequencies. By constructing a sequence of regular processes approximating the
rough white noise, we show that there exists a unique mild solution to the
stochastic direct scattering problem. By studying the expectation and variance
of the solution, we deduce Fredholm integral equations of the first kind for the
inverse problem. It is known that Fredholm integral equations of the first kind
are severely ill-posed, which can be clearly seen from the distribution of
singular values for our integral equations. It is particularly true for the
integral equations of reconstructing the variance. We present well conditioned
integral equations via linear combination of the original equations. We
propose a regularized Kaczmarz method to solve the resulting linear system of
algebraic equations. This method is consistent with our multiple frequency data
and requires solving a relatively small scale system at each iteration.
Numerical experiments show that the proposed approach is effective to solve the
problem.

This work is a nontrivial extension of the method proposed in \cite{BCL} for the
inverse random source scattering problem of the stochastic Helmholtz equation,
to solve the inverse random source scattering problem of the stochastic Navier
equation. Clearly, the elastic wave equation is more challenging due to the
coexistence of compressional waves and shear waves that propagate at different
speeds. The Green function is more complicated and has higher singularity for
the Navier equation than that of the Helmholtz equation does. Hence more
sophisticated analysis is required. 

The paper is organized as follows. In section 2, we introduce the stochastic
Navier equation for elastic waves and discuss the solutions of the
deterministic and stochastic direct problems. Section 3 is devoted to the
inverse problem, where Fredholm integral equations are deduced and the
regularized Kaczmarz method is proposed to reconstruct the mean and the
variance. Numerical experiments are presented in section 4 to illustrate the
performance of the proposed method. The paper is concluded with general
remarks in section 5.

\section{Direct problem}

In this section, we introduce the Navier equation for elastic waves and discuss
the solutions of the deterministic and stochastic direct source scattering
problems.

\subsection{Problem formulation}

Consider the scattering problem of the two-dimensional stochastic Navier
equation in a homogeneous and isotropic medium
\begin{equation}\label{ne}
 \mu\Delta \boldsymbol{u} + (\lambda+\mu) \nabla\nabla\cdot\boldsymbol{u}
+\omega^2 \boldsymbol{u} = \boldsymbol{f}\quad\text{in} ~ \mathbb{R}^2,
\end{equation}
where $\omega>0$ is the angular frequency, $\lambda$ and $\mu$ are the Lam\'{e}
constants
satisfying $\mu>0$ and $\lambda+\mu>0$, and $\boldsymbol{u}=(u_1, u_2)^\top$ is
the displacement of the random wave field. As a source, the electric
current density $\boldsymbol{f}=(f_1, f_2)^\top$ is assumed to be a random
function driven by an additive white noise and takes the
form
\begin{equation}\label{rsf}
 \boldsymbol{f}(x)=\boldsymbol{g}(x)+\boldsymbol{h}(x)\dot{W}_x.
\end{equation}
Here $\boldsymbol{g}=(g_1, g_2)^\top$ is a deterministic real vector function
and $\boldsymbol{h}={\rm diag}(h_1, h_2)$ is a deterministic diagonal
matrix function with $h_j\geq 0$. We assume that $g_j, h_j, j=1, 2$ have
compact supports contained in the rectangular domain $D\subset\mathbb{R}^2$. 
$W(x)=(W_1(x), W_2(x))^\top$ is a two-dimensional two-parameter Brownian sheet
where $W_1(x)$ and $W_2(x)$ are two independent one-dimensional
two-parameter Brownian sheets . $\dot{W}_x$ is a white noise which can be
thought as the derivative of the Brownian sheet $W_x$. To make the paper
self-contained, some preliminaries are presented in the appendix for the
Brownian sheet, white noise, and corresponding stochastic integrals. 

In this random source model, $\boldsymbol{g}$ and $\boldsymbol{h}$ can be viewed
as the mean and standard deviation of $\boldsymbol{f}$, respectively. Hence $
\boldsymbol{h}^2={\rm diag}(h_1^2, h_2^2)$ is the variance of
$\boldsymbol{f}$. To ensure the uniqueness of the solution, the following
Kupradze-Sommerfeld radiation condition is required for the radiated wave
field: 
\begin{equation}\label{rc}
\lim_{r\to\infty}r^{1/2}(\partial_r \boldsymbol{u}_{\rm p}-{\rm i}\kappa_{\rm
p}\boldsymbol{u}_{\rm p})=0,\quad\lim_{r\to\infty}r^{1/2}(\partial_r
\boldsymbol{u}_{\rm s}-{\rm i}\kappa_{\rm s}\boldsymbol{u}_{\rm s})=0,\quad
r=|x|,
\end{equation}
uniformly in all directions $\hat{x}=x/|x|$, where
\[
 \boldsymbol{u}_{\rm p}= -\frac{1}{\kappa^2_{\rm
p}}\nabla\nabla\cdot\boldsymbol{u},\quad \boldsymbol{u}_{\rm
s}=\frac{1}{\kappa^2_{\rm s}}\nabla\times(\nabla\times\boldsymbol{u})
\]
are the compressional component and the shear component of $\boldsymbol{u}$,
respectively, and 
\[
\kappa_{\rm p}=\frac{\omega}{\sqrt{\lambda+2\mu}},  \quad \kappa_{\rm
s}=\frac{\omega}{\sqrt{\mu}}
\]
are known as the compressional wavenumber and the shear wavenumber,
respectively. 

Let $B_R=\{x\in\mathbb{R}^d: |x|<R\}$ be the ball with radius $R$. Denote by
$\partial B_R$ the boundary of $B_R$. Let $R$ be large enough such that
$\bar{D}\subset B_R$. Given the random electric current density function
$\boldsymbol{f}$, i.e., given $\boldsymbol{g}$ and $\boldsymbol{h}$, the direct
problem is to determine the random wave field $\boldsymbol{u}$ of the stochastic
scattering problem \eqref{ne} and \eqref{rc}. The inverse problem is to
determine the mean $\boldsymbol{g}$ and the standard deviation $\boldsymbol{h}$
or the variance $\boldsymbol{h}^2$ of the random source function from the
measured random wave field on $\partial B_R$ at a finite number of 
frequencies $\omega_k, k=1, \dots, K$.

\subsection{Deterministic direct problem}

We begin with the solution for the deterministic direct source problem. Let
$\boldsymbol{h} =0$ in \eqref{rsf}, i.e., no randomness is present in the
source. The stochastic scattering problem reduces to the deterministic
scattering problem:
\begin{equation}\label{dp}
 \begin{cases}
   \mu\Delta \boldsymbol{u} + (\lambda+\mu) \nabla\nabla\cdot\boldsymbol{u}
+\omega^2 \boldsymbol{u}=\boldsymbol{g}\quad& \text{in}~\mathbb{R}^2,\\[2pt]
\partial_r \boldsymbol{u}_{\rm p}-{\rm i}\kappa_{\rm
p}\boldsymbol{u}_{\rm p}=o(r^{-1/2})\quad&\text{as} ~ r\to\infty,\\[2pt]
\partial_r \boldsymbol{u}_{\rm s}-{\rm
i}\kappa_{\rm s}\boldsymbol{u}_{\rm s}=o(r^{-1/2})\quad&\text{as} ~ r\to\infty.
 \end{cases}
\end{equation}
Given $\boldsymbol{g}\in L^2(D)^2$, it is known that the scattering problem
\eqref{dp} has a unique solution
\begin{equation}\label{sdp}
 \boldsymbol{u}(x, \omega)=\int_D \mathbb{G}(x, y,\omega) \boldsymbol{g}(y) {\rm
d}y,
\end{equation}
where $\mathbb{G}$ is the Green tensor function of the Navier equation:
\begin{align*}
\mathbb{G}(x,y,\omega)=\frac{\rm i}{4\mu}H_{0}^{(1)}(\kappa_{\rm
s}|x-y|)\mathbb{I}+\frac{\rm i}{4\omega^2}\nabla_{x}\nabla_{x}^\top
\Bigl[H_{0}^{(1)}(\kappa_{\rm s}|x-y|)-H_{0}^{(1)}(\kappa_{\rm p}|x-y|)\Bigr].
\end{align*}
Here $\mathbb{I}$ is the $2\times 2$ identity matrix, $H_{0}^{(1)}$ is the
Hankel function of the first kind with order zero, and
\[
\nabla_{x}\nabla_{x}^\top=\begin{pmatrix}
  \partial_{x_1x_1} & \partial_{x_1x_2}\\
  \partial_{x_1x_2} & \partial_{x_2x_2}
\end{pmatrix}.
\]

It is known that this Green tensor function has an equivalent form
\begin{equation}\label{gf}
  \mathbb{G}(x,y,\omega)=G_{1}(|x-y|)\mathbb{I}+G_{2}(|x-y|)\mathbb{J}(x-y), 
\end{equation}
which plays an vital role in derivation of the subsequent regularity analysis.
Here for $x\in {\mathbb R}^2\backslash\{0\}$, the matrix $\mathbb{J}$ is given
by
\[
\mathbb{J}(x)=\frac{xx^\top}{|x|^2}.
\]
It is shown in \cite{K-96} that the functions $G_{1}$ and $G_{2}$ can be
decomposed into 
\begin{equation}\label{G_j}
  G_{j}(v)=\frac{1}{\pi}\log (v) \Phi_{j}(v)+\eta_{j}(v),\quad j=1,2,
\end{equation}
where $\Phi_{j}$ and $\eta_{j}$ are analytic functions. Explicitly, we have 
\begin{equation}\label{Phi_j}
  \Phi_{1}(v)=\alpha+\beta_{1}v^2+O(v^4),\quad \Phi_{2}=\beta_{2}v^2+O(v^4)
\end{equation}
and
\begin{equation}\label{eta_j}
  \eta_{1}(v)=\gamma_{1}+O(v^2),\quad \eta_{2}=\gamma_{2}+O(v^2)
\end{equation}
for $v\rightarrow 0$, where $\alpha$, $\beta_{j}$, and $\gamma_{j}$ are
constants depending on $\omega$, $\mu$, and $\lambda$.

The following regularity results of the Green tensor function play an
important role in the analysis of the stochastic direct scattering problem.

\begin{lemma}\label{gfe}
Let $\Omega\subset\mathbb{R}^2$ be a bounded domain. Then $\|\mathbb{G}(\cdot,
y)\|\in L^2(\Omega),\, \forall\, y\in\Omega$, i.e.
\[
\int_{\Omega}\|\mathbb{G}(x,y,\omega)\|^{2} {\rm d}x<\infty,
\]
where $\|\cdot\|$ is the Frobenius norm. 
\end{lemma}

\begin{proof}
Let $a=\sup_{x, y\in\Omega}|x-y|$. We have $\bar{\Omega}\subset B_a(y)$, where
$B_a(y)$ is the ball with radius $a$ and center at $y$. Since $\|J(w)\|=1$, it
follows from the expression in \eqref{gf} and \eqref{G_j} that we only need to
show that
\[
 \log(|x-y|)|x-y|^{m}\in L^2(\Omega),\quad\forall\, y\in\Omega, ~ m\geq
0.
\]
A simple calculation yields
\begin{align*}
 \int_\Omega \Bigl| \log(|x-y|)|x-y|^{m} \Bigr|^2 {\rm d}x&\leq
\int_{B_a(y)} \Bigl|\log(|x-y|)|x-y|^{m} \Bigr|^2 {\rm d}x\\
&\lesssim \int_0^a r^{1+2m}\Bigl|\log(r)\Bigr|^2{\rm d}r<\infty,
\end{align*}
which completes the proof.
\end{proof}

Throughout the paper, $a\lesssim b$ stands for $a\leq C b$, where $C$ is a
positive constant and its specific value is not required but should be clear
from the context.

\begin{lemma}
Let $\Omega\subset\mathbb{R}^2$ be a bounded domain. We have for
$\alpha\in(\frac{3}{2}, \, \infty)$ that

\begin{equation}\label{G2}
 \int_\Omega \|\mathbb{G}(x, y,\omega)-\mathbb{G}(x, z,\omega)\|^\alpha {\rm
d}x\lesssim |y-z|^{\frac{3}{2}},\quad\forall ~ y, z\in\Omega.
\end{equation}

\end{lemma}

\begin{proof}
It follows from \eqref{gf} and the triangle inequality that 
\begin{align*}
   &\int_\Omega \|\mathbb{G}(x, y,\omega)-\mathbb{G}(x, z,\omega)\|^\alpha {\rm
d}x\\ 
&\lesssim\int_{\Omega}\Bigl\|G_{1}(|x-y|)\mathbb{I}-G_{1}(|x-z|)\mathbb{I}
\Bigr\|^{\alpha}{\rm
d}x\\
&\hspace{4cm}+\int_{\Omega}\Bigl\|G_{2}(|x-y|)\mathbb{J}(x-y)-G_{2}
(|x-z|)\mathbb{J}(x-z)\Bigr\|^{\alpha}{\rm d}x\\
   &\lesssim \int_{\Omega}\Bigl|G_{1}(|x-y|)-G_{1}(|x-z|)\Bigr|^{\alpha}{\rm d}x
 +\int_{\Omega}\Bigl|G_{2}(|x-y|)-G_{2}(|x-z|)\Bigr|^{\alpha}\|\mathbb{J}
(x-y)\|^{\alpha}{\rm d}x\\
 &\hspace{4cm}+\int_{\Omega}|G_{2}(|x-z|)|^{\alpha}\|\mathbb{J}(x-y)-\mathbb{J}
(x-z)\|^{\alpha}{\rm d}x\\
   &=:T_1+T_2+T_3.
\end{align*}
We shall only present the estimates of $T_1$ and $T_3$ since the estimates of
$T_1$ and $T_2$ are similar due to  $\|\mathbb{J}(x)\|=1$ and \eqref{G_j}.

Step 1. The estimate of $T_1$. It suffices to estimate the singular part $G_{1}$
of $T_1$. It follows from \eqref{G_j}, \eqref{Phi_j}, and \eqref{eta_j} that we
have
\begin{align*}
  T_1\lesssim & \int_{\Omega}\Bigl|\log(|x-y|)-\log(|x-z|)\Bigr|^{\alpha}{\rm
d}x+\int_{\Omega}\Bigl||x-y|^{p}-|x-z|^{p}\Bigr|^{\alpha}{\rm d}x\\ 
&+\int_{\Omega}\Bigl|\log(|x-y|)|x-y|^{p}-\log(|x-z|)|x-z|^{p}\Bigr|^{\alpha}{
\rm d}x\\
  =:& T_1^{1}+T_1^{2}+T_1^{3},
\end{align*}
where $p=2$ or $4$. 

For the term $T_1^{1}$, we have
\begin{align*}
T_1^{1}=&\int_\Omega \Bigl||x-y|-|x-z|\Bigr|^{\frac{3}{2}}
\Bigl|\log(|x-y|)-\log(|x-z|)\Bigr|^{\alpha-\frac{3}{2}}\\
&\hspace{4cm}\times\left|\int_0^1 \Bigl(|x-y|t+|x-z|(1-t)\Bigr)^{-1}{\rm d}t
\right|^{\frac{3}{2}}{\rm d}x\\
\leq&|y-z|^{\frac{3}{2}}\int_\Omega
 \Bigl|\log(|x-y|)-\log(|x-z|)\Bigr|^{\alpha-\frac{3}{2}}
 \left|\frac{1}{|x-y|}+\frac{1}{
|x-z|} \right|^{\frac{3}{2}}{\rm d}x\\
\leq& |y-z|^{\frac{3}{2}}\left(\int_\Omega
\Bigl(\frac{1}{|x-y|}+\frac{1}{|x-z|}\Bigr)^{\frac{9
}{5}}{\rm d}x \right)^{\frac{5}{6}}\\
&\hspace{4cm}\times\left(\int_\Omega
\Bigl|\log(|x-y|)-\log(|x-z|)\Bigr|^{6\alpha-9}{\rm
d}x \right)^{\frac{1}{6}},
\end{align*}
where we have used the H\"older inequality in the last step.

Let $a=\sup_{x, y\in\Omega}|x-y|$ and $b=\sup_{x, z\in\Omega}|x-z|$. We have
$\bar{\Omega}\subset B_a(y)$ and $\bar{\Omega}\subset B_b(z)$, where $B_a(y)$
and $B_b(z)$ are the discs with radii $a$ and $b$ and centers at $y$ and $z$,
respectively. It is easy to verify that
\begin{align*}
 \int_\Omega \Bigl(\frac{1}{|x-y|}+\frac{1}{|x-z|}
\Bigr)^{\frac{9}{5}}{\rm d}x &\lesssim \int_{B_a(y)}
\frac{1}{|x-y|^{\frac{9}{5}}}{\rm d}x +\int_{B_b(z)}
\frac{1}{|x-z|^{\frac{9}{5}}}{\rm d}x \\
&\lesssim \int_0^a r^{-\frac{4}{5}}{\rm d}r +\int_0^b r^{-\frac{4}{5}}{\rm d}r
<\infty
\end{align*}
and
\begin{align*}
&\int_\Omega\Bigl|\log(|x-y|)-\log(|x-z|)\Bigr|^{6\alpha-9}{\rm
d}x \\
&\lesssim \int_{B_a(y)}\Bigl|\log(|x-y|)\Bigr|^{6\alpha-9}{\rm
d}x+\int_{B_b(z)}\Bigl|\log(|x-z|)\Bigr|^{6\alpha-9}{\rm
d}x\\
&\lesssim \int_0^a r\Bigl|\log(r)\Bigr|^{6\alpha-9}{\rm d}r+
\int_0^b r\Bigl|\log(r)\Bigr|^{6\alpha-9}{\rm d}r<\infty.
\end{align*}
Combining the above estimates gives $T_1^{1}\lesssim |y-z|^{\frac32}$.

For the term $T_1^{2}$, using the identity
\[
a^{p}-b^{p}=(a-b)\bigl(a^{p-1}+a^{p-2}b+\cdots+ab^{p-2}+b^{p-1}\bigr),
\]
we obtain 
\begin{align*}
T_1^2&\lesssim |y-z|^{\alpha}\int_{\Omega}\Bigl[|x-y|^{(p-1)\alpha}+|x-y|^{
(p-2)\alpha}|x-z|^{ \alpha}
  +\cdots+|x-z|^{(p-1)\alpha}\Bigr]{\rm d}x\\
  &\lesssim |y-z|^{\alpha}.
\end{align*}
Applying the technique of estimating the terms $T_1^1$ and $T_1^2$, we
get the estimate for the term $T_1^{3}$:
\begin{align*}
T_1^{3}&\lesssim
\int_{\Omega}\Bigl|\log(|x-y|)\Bigr|^{\alpha}\Bigl||x-y|^{p}-|x-z|^{p}\Bigr|^{
\alpha}{\rm d}x\\
&\hspace{4cm}+\int_{\Omega}\Bigl|\log(|x-y|)-\log(|x-z|)\Bigr|^{\alpha}|x-z|^{
p\alpha}{\rm d}x\\
&\lesssim|y-z|^{\alpha}\int_{\Omega}\Bigl|\log(|x-y|)\Bigr|^{\alpha}\Bigl[|x-y|^
{(p-1)\alpha}+|x-y|^{(p-2)\alpha}|x-z|^{\alpha}
  +\cdots+|x-z|^{(p-1)\alpha}\Bigr]{\rm d}x\\
  &\quad+
|y-z|^{\frac32}\int_{\Omega}\Bigl|\log(|x-y|)-\log(|x-z|)\Bigr|^{\alpha-\frac{3}
{2}}
 \left|\frac{1}{|x-y|}+\frac{1}{
|x-z|} \right|^{\frac{3}{2}}|x-z|^{p\alpha}{\rm d}x\\
&\lesssim |y-z|^{\alpha}+|y-z|^{\frac32}.
\end{align*}
Therefore, we obtain $T_1\leq |y-z|^{\frac32}$ for $\alpha>\frac32$. 
Similarly, we may have the estimate $T_2\leq |y-z|^{\frac32}$.

Step 2. The estimate of term $T_3$. Recall 
\[
\mathbb{J}(x)=\begin{bmatrix}J_{11}(x) & J_{12}(x)\\ J_{12}(x) &
J_{22}(x)\end{bmatrix}
=\begin{bmatrix}\frac{x_{1}^{2}}{|x|^2} & \frac{x_1x_2}{|x|^2}\\
\frac{x_1x_2}{|x|^2} & \frac{x_2^2}{|x|^2}\end{bmatrix}.
\]
A simple calculation yields
\begin{align*}
\|\mathbb{J}(x-y)-\mathbb{J}(x-z)\|^2&=|J_{11}(x-y)-J_{11}(x-z)|^2\\
&+2|J_{12}(x-y)-J_{12}(x-z)|^2+|J_{22}(x-y)-J_{22}(x-z)|^2. 
\end{align*}
It is easy to verify that 
\begin{align*} 
&|J_{11}(x-y)-J_{11}(x-z)|\\
=&\left|\frac{(x_1-y_1)^2}{|x-y|^2}-\frac{(x_1-z_1)^2}
{ |x-z|^2}\right|\\
=&\left|\Big(\frac{x_1-y_1}{|x-y|}+\frac{x_1-z_1}{|x-z|}\Big)\Big(\frac{x_1-y_1}
{|x-y|}-\frac{x_1-z_1}{|x-z|}\Big)\right|\\
  \lesssim& \left|\frac{x_1-y_1}{|x-y|}-\frac{x_1-z_1}{|x-z|}\right|\\
=&\left|\frac{\big(|x-z|-|x-y|\big)(x_1-y_1)+|x-y|\big((x_1-y_1)-(x_1-z_1)\big)}
{|x-y||x-z|}\right|\\
  \lesssim& \frac{|y-z|}{|x-z|}.
\end{align*}
Similarly, we may show that 
\begin{align*}
|J_{22}(x-y)-J_{22}(x-z)|=\left|\frac{(x_2-y_2)^2}{|x-y|^2}-\frac{(x_2-z_2)^2}{
|x-z|^2}\right|
  \lesssim \frac{|y-z|}{|x-z|}
\end{align*}
and
\begin{align*} 
|J_{12}(x-y)-J_{12}(x-z)|=&\left|\frac{(x_1-y_1)(x_2-y_2)}{|x-y|^2}-\frac{
(x_1-z_1)(x_2-z_2)}{|x-z|^2}\right|\\
  \leq &
\left|\frac{(x_1-y_1)(x_2-y_2)}{|x-y|^2}-\frac{(x_1-y_1)(x_2-y_2)}{|x-z|^2}
\right|\\
& 
+\left|\frac{(x_1-y_1)(x_2-y_2)}{|x-z|^2}-\frac{(x_1-z_1)(x_2-z_2)}{|x-z|^2}
\right|\\
\leq & \frac{|y-z|}{|x-z|}+\frac{|x-y|}{|x-z|^2}|y-z|.
\end{align*}
Hence 
\[
\|\mathbb{J}(x-y)-\mathbb{J}(x-z)\|\lesssim
\frac{|y-z|}{|x-z|}+\frac{|x-y|}{|x-z|^2}|y-z|
\]
and 
\[
\|\mathbb{J}(x-y)-\mathbb{J}(x-z)\|\lesssim
\frac{|y-z|}{|x-y|}+\frac{|x-z|}{|x-y|^2}|y-z|.
\]

Again, it suffices to consider the singular part $G_2$ of $T_3$.
Using \eqref{G_j}, \eqref{Phi_j}, and \eqref{eta_j}, we split the term
$T_3$ into two parts:
\begin{align*}
 T_3\lesssim &
\int_{\Omega}\left|\log(|x-z|)\Big[|x-z|^2+|x-z|^4\Big]+|x-z|^2\right|^{\alpha}
\|\mathbb{J}(x-y)-\mathbb{J}(x-z)\|^{\alpha}{\rm d}x\\
  &+\int_{\Omega}\|\mathbb{J}(x-y)-\mathbb{J}(x-z)\|^{\alpha}{\rm d}x\\
  =:&T_3^1+T_3^2.
\end{align*}
For the term $T_3^1$, we have
\begin{align*}
  T_3^1\lesssim &
|y-z|^{\alpha}\int_{\Omega}\left|\log(|x-z|)\Big[|x-z|^2+|x-z|^4\Big]
+|x-z|^2\right|^{\alpha}
  \Big[\frac{1}{|x-z|^{\alpha}}+\frac{|x-y|^{\alpha}}{|x-z|^{2\alpha}}\Big]{\rm
d}x\\
  \lesssim & |y-z|^{\alpha}.
\end{align*}
The estimate of $T_3^2$ is more technical. We let $\xi=\frac{y+z}{2}$ and 
$r=|y-z|$. It is clear to note that 
\begin{align*}
T_3^2&=\left(\int_{B_{\frac{r}{4}}(y)}+\int_{B_{\frac{r}{4}}(z)}+\int_{B_{2r}
(\xi)\backslash B_{\frac{r}{4}}(y) \cup B_{\frac{r}{4}}(z) }
  +\int_{\Omega\backslash B_{2r}(\xi)}\right)
\|\mathbb{J}(x-y)-\mathbb{J}(x-z)\|^{\alpha}{\rm d}x\\
 & =:I_1+I_2+I_3+I_4.
\end{align*}
Next we estimate the above four parts. First we have 
\[
  I_1 \lesssim \int_{B_{\frac{r}{4}}(y)}
\frac{r^{\alpha}}{|x-z|^{\alpha}}+\frac{|x-y|^{\alpha}}{|x-z|^{2\alpha}}r^{
\alpha} {\rm d}x\lesssim \int_{B_{\frac{r}{4}}(y)} {\rm d}x\lesssim r^2,
\]
where we have utilized the fact that $|x-y|\leq
\frac{r}{4}$ and $|x-z|>\frac{r}{2}$ for $x\in B_{\frac{r}{4}}(y)$. Similarly,
we have
\begin{align*}
I_2 \lesssim \int_{B_{\frac{r}{4}}(z)}
\frac{r^{\alpha}}{|x-y|^{\alpha}}+\frac{|x-z|^{\alpha}}{|x-y|^{2\alpha}}r^{
\alpha} {\rm d}x
\lesssim r^2.
\end{align*}
For the term $I_3$, we have $|x-z|\geq \frac{r}{4}$ and
$\frac{r}{4}\leq|x-y|<3r$  for any $x\in B_{2r}(\xi)\backslash
B_{\frac{r}{4}}(y) \cup B_{\frac{r}{4}}(z)$. Thus
\begin{align*}
I_3\leq \int_{B_{2r}(\xi)\backslash B_{\frac{r}{4}}(y) \cup
B_{\frac{r}{4}}(z)}\frac{r^{\alpha}}{|x-z|^{\alpha}}+\frac{|x-y|^{\alpha}}{
|x-z|^{2\alpha}}r^{ \alpha} {\rm d}x
  \lesssim \int_{B_{2r}(\xi)\backslash B_{\frac{r}{4}}(y) \cup
B_{\frac{r}{4}}(z) } {\rm d}x\lesssim r^2.
\end{align*}
It is clear to note for $x\in \Omega\backslash B_{2r}(\xi)$ that
\[
\left|\frac{|x-\xi|}{|x-y|}-1\right|\leq \frac{|\xi-y|}{|x-y|}\leq
\frac{r/2}{2r-r/2}=\frac{1}{3}. 
\]
We have 
\[
\frac23 |x-y|\leq |x-\xi|\leq \frac43|x-y|
\]
and
\[
\frac23|x-z|\leq |x-\xi|\leq \frac43|x-z|,
\]
which give
\begin{align*}
I_4 &\lesssim \int_{\Omega\backslash B_{2r}(\xi)}
\frac{r^{\alpha}}{|x-z|^{\alpha}}+\frac{|x-y|^{\alpha}}{|x-z|^{2\alpha}}r^{
\alpha} {\rm d}x\\
& \lesssim \int_{\Omega\backslash B_{2r}(\xi)}
\frac{r^{\alpha}}{|x-\xi|^{\alpha}}   {\rm d}x\\
&\lesssim r^{\alpha}\int_{2r}^{R}s^{1-\alpha}ds\lesssim
r^{\alpha}\Big[R^{2-\alpha}+(2r)^{2-\alpha}\Big]\lesssim r^{\alpha}+r^{2},
\end{align*}
where $\bar{\Omega}\subset B_{R}(\xi)$ with $B_{R}(\xi)$ being the disc with
radius $R$ and center at $\xi$. Combining all the estimates in step 2 gives
$T_3\lesssim |y-z|^{\min\{\alpha,\,2\}}$.

The proof is completed by combining step 1 and step 2.
\end{proof}

\subsection{Stochastic direct problem}

In this section, we discuss the solution for the stochastic direct source
scattering problem:
\begin{equation}\label{sp}
 \begin{cases}
   \mu\Delta \boldsymbol{u} + (\lambda+\mu) \nabla\nabla\cdot\boldsymbol{u}
+\omega^2 \boldsymbol{u}=\boldsymbol{g}+ \boldsymbol{h}\dot{W}_x\quad&
\text{in}~\mathbb{R}^2,\\[2pt]
\partial_r \boldsymbol{u}_{\rm p}-{\rm i}\kappa_{\rm
p}\boldsymbol{u}_{\rm p}=o(r^{-1/2})\quad &\text{as} ~ r\to\infty,\\[2pt]
\partial_r \boldsymbol{u}_{\rm s}-{\rm i}\kappa_{\rm
s}\boldsymbol{u}_{\rm s}=o(r^{-1/2})\quad & \text{as}~ r\to\infty.
 \end{cases}
\end{equation}

Let us first specify the regularity of $\boldsymbol{g}$ and $\boldsymbol{h} $
before discussing the solution of the stochastic scattering problem \eqref{sp}.
Motivated by the solution of the deterministic direct problem \eqref{dp}, we
assume that $\boldsymbol{g}\in L^2(D)^2$. The regularity of $\boldsymbol{h}$ is
chosen such that the stochastic integral
\[
 \int_D \mathbb{G}(x, y, \omega) \boldsymbol{h} (y) {\rm d}W_y
\]
satisfies
\begin{align*}
 {\bf E}\Bigl(\bigl|\int_D \mathbb{G}(x, y, \omega) \boldsymbol{h} (y){\rm
d}W_y\bigr|^2\Bigr)=\int_D \|\mathbb{G}(x, y, \omega)\boldsymbol{h}(y)\|^2  {\rm
d}y\\
\leq \int_{D}\|\mathbb{G}(x,y,\omega)\|^2\|\boldsymbol{h}(y)\|^{2}{\rm
d}x<\infty,
\end{align*}
where Proposition \ref{prop} is used in the above identity.

We only need to consider the singular part of the Green tensor function. It
follows from the H\"{o}lder inequality that
\[
 \int_D \Bigl|\log(|x-y|)|x-y|^{m}\Bigr|^2  \|\boldsymbol{h}(y)\|^{2}{\rm
d}y\leq\Bigl(\int_D \Bigl|\log(|x-y|)|x-y|^{m}\Bigr|^{\frac{2p}{p-2}} {\rm
d}y\Bigr)^{\frac{p-2}{p}} \left(\int_D \| \boldsymbol{h} (y)\|^p {\rm
d}y\right)^{\frac{2}{p}}.
\]
Since the first term on the right-hand side of the above inequality is a
singular integral, $p$ should be chosen such that it is well defined. Let
$\rho>0$ be sufficiently large such that $\bar{D}\subset B_\rho(x)$, where
$B_\rho(x)$ is the disc with radius $\rho$ and center at $x$. A simple
calculation yields
\begin{align*}
\int_D \Bigl|\log(|x-y|)|x-y|^{m}\Bigr|^{\frac{2p}{p-2}} {\rm
d}y\leq \int_{B_\rho(x)}\Bigl|\log(|x-y|)|x-y|^{m}\Bigr|^{\frac{2p}{p-2}} {\rm
d}y\\
\lesssim
\int_0^\rho r^{1+\frac{2mp}{p-2}}\Bigl|\log(r)\Bigr|^{\frac{2p}{p-2}}{\rm
d}r.
\end{align*}
It is clear to note that the above integral is well defined when $p>2$.

From now on, we assume that $h_j \in L^p(D), j=1, 2$ where $p\in (2, \infty]$.
Moreover, we require that $h_j \in C^{0, \eta}(D)$, i.e., $\eta$-H\"{o}lder
continuous, where $\eta\in(0, 1]$. The H\"{o}lder continuity will be used in the
analysis for existence of the solution.

The following theorem shows the well-posedness of the solution for the
stochastic scattering problem \eqref{sp}. The explicit solution will be used
to derive Fredholm integral equations for the inverse problem.

\begin{theorem}
 Let $\Omega\subset\mathbb{R}^2$ be a bounded domain. There exists a
unique continuous stochastic process $\boldsymbol{u}: \Omega\to\mathbb{C}$
satisfying
\begin{equation}\label{ms}
\boldsymbol{u}(x, \omega)=\int_D \mathbb{G}(x, y, \omega)\boldsymbol{g}(y){\rm
d}y+\int_D \mathbb{G}(x, y, \omega) \boldsymbol{h} (y){\rm d}W_y,\quad a.s.,
\end{equation}
which is called the mild solution of the stochastic scattering problem
\eqref{sp}.
\end{theorem}

\begin{proof}
 First we show that there exists a continuous modification of the random
field
\[
 \boldsymbol{v}(x, \omega)=\int_D \mathbb{G}(x, y, \omega) \boldsymbol{h}
(y){\rm d}W_y,\quad x\in\Omega.
\]
For any $x, z\in\Omega$, we have from Proposition \ref{prop} and the H\"{o}lder
inequality that
\begin{align*}
 {\bf E}(|\boldsymbol{v}(x, \omega)-\boldsymbol{v}(z, \omega)|^2)&=\int_D
\Big\|\Big(\mathbb{G}(x, y, \omega)-\mathbb{G}(z, y, \omega)\Big)
\boldsymbol{h}(y)\Big\|^{2}{\rm d}y\\
 &\leq \left(\int_D \|\mathbb{G}(x, y, \omega)-\mathbb{G}(z, y,
\omega)\|^{\frac{2p}{p-2}}{\rm d}y\right)^{\frac{p-2}{p}} \left(\int_D\|
\boldsymbol{h} (y)\|^p {\rm d}y\right)^{\frac{2}{p}}.
\end{align*}
For $p>2$, it follows from \eqref{G2} that
\[
 \int_D \|\mathbb{G}(x, y, \omega)-\mathbb{G}(z, y,
\omega)\|^{\frac{2p}{p-2}}{\rm d}y\lesssim|x-z|^{\frac{3}{2}},
\]
which gives
\[
 {\bf E}(|\boldsymbol{v}(x, \omega)-\boldsymbol{v}(z, \omega)|^2)\lesssim  \|
\boldsymbol{h} \|^2_{L^p(D)^2} |x-z|^{\frac{3p-6}{2p}}.
\]
Since $\boldsymbol{v}(x, \kappa)-\boldsymbol{v}(z, \kappa)$ is a random Gaussian
variable, we have (cf. \cite[Proposition 3.14]{H-09}) for any integer $q$ that
\[
 {\bf E}(|\boldsymbol{v}(x, \omega)-\boldsymbol{v}(z, \omega)|^{2q})\lesssim
\left({\bf E}(|\boldsymbol{v}(x,
\omega)-\boldsymbol{v}(z, \omega)|^2)\right)^q\lesssim\| \boldsymbol{h}
\|^{2q}_{L^p(D)^2}|x-z|^{\frac{q(3p-6)}{2p}}.
\]
Taking $q>\frac{2p}{3p-6}$, we obtain from Kolmogorov's
continuity theorem that there exists a P-a.s. continuous modification of the
random field $\boldsymbol{v}$.

Clearly, the uniqueness of the  mild solution comes from the solution
representation formula \eqref{ms}, which depends only on the Green function
$\mathbb{G}$ and the source functions $\boldsymbol{g}$ and $\boldsymbol{h}$.

Next we present a constructive proof to show the existence. We construct a
sequence of processes $\dot{W}^n_x$ satisfying $\boldsymbol{h} \dot{W}^n\in
L^2(D)^2$ and a sequence
\[
 \boldsymbol{v}^n(x, \omega)=\int_D \mathbb{G}(x, y, \omega) \boldsymbol{h}
(y){\rm d}W^n_y,\quad x\in\Omega,
\]
which satisfies $\boldsymbol{v}^n\to \boldsymbol{v}$ in $L^2(\Omega)$ a.s. as
$n\to\infty$.

Let $\mathcal{T}_n=\cup_{j=1}^n K_j$ be a regular triangulation of $D$, where
$K_j$ are triangles. Denote
\[
 \xi_j=|K_j|^{-\frac{1}{2}}\int_{K_j} {\rm d}W_x,\quad 1\leq j\leq n,
\]
where $|K_j|$ is the area of the $K_j$. It is known in \cite{CZZ-JCM08} that
$\xi_j$ is a family of independent identically distributed normal random
variables with mean zero and variance one. The piecewise constant
approximation sequence is given by
\[
 \dot{W}^n_x=\sum_{j=1}^n |K_j|^{-\frac{1}{2}}\xi_j \chi_j(x),
\]
where $\chi_j$ is the characteristic function of $K_j$. Clearly we have for any
$p\geq 1$ that
\begin{align*}
 {\bf E}\bigl(\|\dot{W}^n\|^p_{L^p(D)^2}\bigr)&={\bf E}\Bigl(\int_D
\Bigl|\sum_{j=1}^n |K_j|^{-\frac{1}{2}}\xi_j\chi_j(x)\Bigr|^p {\rm
d}x\Bigr)\lesssim {\bf E}\Bigl(\int_D \sum_{j=1}^n |K_j|^{-\frac{p}{2}}|\xi_j|^p
\chi_j(x){\rm d}x\Bigr)\\
&=\sum_{j=1}^n {\bf E}(|\xi_j|^p)|K_j|^{1-\frac{p}{2}} <\infty,
\end{align*}
which shows that $\dot{W}^n\in L^p(D)^2, p\geq 1$. It follows from the
H\"{o}lder inequality that $\boldsymbol{h} \dot{W}^n\in L^2(D)^2$.

Using Proposition \ref{prop}, we have that
\begin{align*}
 &{\bf E}\Bigl(\int_\Omega \Bigl|\int_D \mathbb{G}(x, y, \omega) \boldsymbol{h}
(y){\rm d}W_y-\int_D \mathbb{G}(x, y, \omega) \boldsymbol{h} (y) {\rm
d}W^n_y\Bigr|^2 {\rm d}x\Bigr)\\
&={\bf E}\Bigl(\int_\Omega \Bigl|\sum_{j=1}^n\int_{K_j} \mathbb{G}(x,
y, \omega) \boldsymbol{h} (y){\rm d}W_y-\sum_{j=1}^n |K_j|^{-1}
\int_{K_j}\mathbb{G}(x, z, \omega) \boldsymbol{h} (z){\rm d}z \int_{K_j} {\rm
d}W_y\Bigr|^2 {\rm d}x\Bigr)\\
&={\bf E}\Bigl(\int_\Omega\Bigl|\sum_{j=1}^n \int_{K_j}\int_{K_j}|K_j|^{-1}
(\mathbb{G}(x, y, \omega) \boldsymbol{h} (y)-\mathbb{G}(x, z, \omega)
\boldsymbol{h} (z)) {\rm d}z{\rm d}W_y\Bigr|^2{\rm d}x\Bigr)\\
&=\int_\Omega\Bigl(\sum_{j=1}^n\int_{K_j}\Bigl||K_j|^{-1}\int_{K_j}(\mathbb{G}(x
,y, \omega) \boldsymbol{h} (y)-\mathbb{G}(x, z, \omega) \boldsymbol{h} (z)){\rm
d}z \Bigr|^2 {\rm d}y\Bigr){\rm d}x\\
&\leq \int_\Omega \Bigl(\sum_{j=1}^n |K_j|^{-1}\int_{K_j}\int_{K_j}
\|\mathbb{G}(x, y, \omega) \boldsymbol{h} (y)-\mathbb{G}(x, z, \omega)
\boldsymbol{h} (z)\|^2{\rm d}z{\rm d}y\Bigr){\rm
d}x\\
&=\sum_{j=1}^n |K_j|^{-1}\int_{K_j}\int_{K_j}\int_\Omega
\|\mathbb{G}(x, y, \omega) \boldsymbol{h} (y)-\mathbb{G}(x, z, \omega)
\boldsymbol{h} (z)\|^2 {\rm d}x {\rm d}z{\rm d}y.
\end{align*}
Using the triangle and Cauchy-Schwartz inequalities, we get
\begin{align*}
 \int_\Omega \|\mathbb{G}(x, y, \omega) \boldsymbol{h} (y)-\mathbb{G}(x, z,
\omega) \boldsymbol{h} (z)\|^2{\rm d}x\lesssim \int_\Omega &\|\mathbb{G}(x, y,
\omega)-\mathbb{G}(x, z, \omega)\|^2\| \boldsymbol{h} (y)\|^2{\rm
d}x\\
&+\int_\Omega \|\mathbb{G}(x, z, \omega)\|^2 \| \boldsymbol{h} (y)-
\boldsymbol{h} (z)\|^2 {\rm d}x.
\end{align*}

It follows from \eqref{G2}, Lemma \ref{gfe}, and the
$\eta$-H\"{o}lder continuity of $ \boldsymbol{h} $ that
\[
 \int_\Omega \|\mathbb{G}(x, y, \omega) \boldsymbol{h} (y)-\mathbb{G}(x, z,
\omega) \boldsymbol{h} (z)\|^2{\rm d}x\lesssim 
\|\boldsymbol{h}(y)\|^{2}|y-z|^{\frac{3}{2}}+|y-z|^{2\eta},
\]
which gives
\begin{align*}
 &{\bf E}\Bigl(\int_\Omega \Bigl|\int_D \mathbb{G}(x, y, \omega) \boldsymbol{h}
(y){\rm d}W_y-\int_D \mathbb{G}(x, y, \omega) \boldsymbol{h} (y)
{\rm d}W^n_y\Bigr|^2 {\rm d}x\Bigr)\\
&\lesssim
\sum_{j=1}^n |K_j|^{-1}\int_{K_j}\int_{K_j}
\|\boldsymbol{h}(z)\|^{2}|y-z|^{\frac{3}{2}}{\rm
d}z{\rm d}y+\sum_{j=1}^n |K_j|^{-1}\int_{K_j}\int_{K_j} |y-z|^{2\eta}{\rm
d}z{\rm d}y\\
&\leq \|\boldsymbol{h} \|^2_{L^2(D)^2}\max_{1\leq j\leq n}({\rm
diam}K_j)^{\frac{3}{2}}+|D| \max_{1\leq j\leq n}({\rm
diam}K_j)^{2\eta}\to 0
\end{align*}
as $n\to\infty$ since the diameter of $K_j\to 0$ as $n\to\infty$.

For each $n\in\mathbb{N}$, we consider the scattering problem:
\begin{equation}\label{shen}
 \begin{cases}
   \mu\Delta \boldsymbol{u}^{n} + (\lambda+\mu)
\nabla\nabla\cdot\boldsymbol{u}^{n} +\omega^2
\boldsymbol{u}^{n}=\boldsymbol{g}+ \boldsymbol{h}\dot{W}_x^{n}\quad&
\text{in}~\mathbb{R}^2,\\[2pt]
\partial_r \boldsymbol{u}_{\rm p}^{n}-{\rm i}\kappa_{\rm
p}\boldsymbol{u}_{\rm p}^{n}=o(r^{-1/2})\quad & \text{as} ~r\to\infty,\\[2pt]
\partial_r \boldsymbol{u}_{\rm s}^{n}-{\rm i}\kappa_{\rm
s}\boldsymbol{u}_{\rm s}^{n}=o(r^{-1/2})\quad &\text{as} ~ r\to\infty. 
 \end{cases}
\end{equation}
It follows from $ \boldsymbol{h} \dot{W}^n_x\in L^2(D)^2$ that the problem
\eqref{shen} has a unique solution which is given by
\begin{equation}\label{shn}
 \boldsymbol{u}^n(x, \omega)=\int_D \mathbb{G}(x, y,
\omega)\boldsymbol{g}(y){\rm d}y+\int_D
\mathbb{G}(x, y, \omega) \boldsymbol{h} (y){\rm d}W^n_y.
\end{equation}
Since ${\bf E}(\|\boldsymbol{v}^n-\boldsymbol{v}\|^2_{L^2(\Omega)^2})\to 0$ as
$n\to\infty$, there exists a
subsequence still denoted as $\{\boldsymbol{v}^n\}$ which converges to
$\boldsymbol{v}$ a.s.. Letting $n\to\infty$ in \eqref{shn}, we obtain the mild
solution \eqref{ms} and complete the proof.
\end{proof}

It is clear to note that the mild solution of the stochastic direct problem
\eqref{ms} reduces to the solution of the deterministic direct problem
\eqref{sdp} when $ \boldsymbol{h} =0$, i.e., no randomness is present in the
source.

\section{Stochastic inverse problem}

In this section, we derive the Fredholm integral equations and present
a regularized Kaczmarz method to solve the stochastic inverse problem
by using multiple frequency data.

\subsection{Integral equations}

Recall the mild solution of the stochastic direct scattering problem at
angular frequency $\omega_k$:
\begin{equation}\label{sd}
\boldsymbol{u}(x, \omega_k)=\int_D \mathbb{G}(x, y,
\omega_k)\boldsymbol{g}(y){\rm d}y + \int_D
\mathbb{G}(x, y, \omega_k) \boldsymbol{h} (y){\rm d}W_y.
\end{equation}
Taking the expectation on both sides of \eqref{sd} and using the identity 
\[
 {\bf E}\Bigl(\int_D \mathbb{G}(x, y, \omega_k) \boldsymbol{h} (y){\rm
d}W_y\Bigr)=0,
\]
we obtain
\[
 {\bf E}(\boldsymbol{u}(x, \omega_k))=\int_D \mathbb{G}(x, y,
\omega_k)\boldsymbol{g}(y){\rm d}y,
\]
which is a complex-valued Fredholm integral equation of the first kind and may
be used to reconstruct $\boldsymbol{g}$. However, it is more convenient to solve
real-valued equations. We shall split all the complex-valued quantities into
their real and imaginary parts, which also allows us to deduce the equations for
the variance.

Let $\boldsymbol{u}={\rm Re}\boldsymbol{u}+{\rm i}{\rm Im}\boldsymbol{u}$ and
$\mathbb{G}={\rm Re}\mathbb{G}+{\rm i}{\rm Im}\mathbb{G}$. Using \eqref{gf}, we
have more explicit formulas of $G_1$ and $G_2$:
\[
G_{1}(v)=\frac{\rm i}{4\mu}H_{0}^{(1)}(\kappa_{\rm s}v)-\frac{\rm i}{4\omega^2v}
\Big[\kappa_{\rm s}H_{1}^{(1)}(\kappa_{\rm s}v)-\kappa_{\rm
p}H_{1}^{(1)}(\kappa_{\rm p} v)\Big]
\]
and
\[
G_{2}(v)=\frac{\rm
i}{4\omega^2}\Big[\frac{2\kappa_{\rm
s}}{v}H_{1}^{(1)}(\kappa_{\rm s}v)-\kappa_{\rm s}^{2}
H_{0}^{(1)}(\kappa_{\rm s}v)-\frac{2\kappa_{\rm
p}}{v}H_{1}^{(1)}(\kappa_{\rm p}v)+
\kappa_{\rm p}^{2} H_{0}^{(1)}(\kappa_{\rm p}v)\Big],
\]
where $H_{0}^{(1)}$ and $H_{1}^{(1)}$ are the Hankel functions of the first
kind with order zero and one, respectively.

Denote 
\[
\kappa_{{\rm p}, k}=\frac{\omega_k}{\sqrt{\lambda+2\mu}}, \quad
\kappa_{{\rm s}, k}=\frac{\omega_k}{\sqrt{\mu}}.
\]
Let
\[
  {\rm Re}\mathbb{G}(x,y,\omega_k)=\begin{bmatrix}
    G_{\rm Re}^{[11]}(x,y,\omega_k)  &  G_{\rm Re}^{[12]}(x,y,\omega_k)\\[2pt]
    G_{\rm Re}^{[21]}(x,y,\omega_k)  &  G_{\rm Re}^{[22]}(x,y,\omega_k)
  \end{bmatrix},
\]
where
\[
  \begin{split}
G_{\rm
Re}^{[11]}(x,y,\omega_k)=&-\frac{1}{4\mu}Y_{0}(\kappa_{{\rm s},k}|x-y|)+\frac{1}
{4\omega_k^2|x-y|}\Bigl(\kappa_{{\rm s}, k}Y_{1}(\kappa_{{\rm
s},k}|x-y|)\\
&-\kappa_{{\rm p},k}Y_1(\kappa_{{\rm
p},k}|x-y|)\Bigr)-\frac{(x_1-y_1)^2}{4\omega_k^2|x-y|^2}\Bigl[\frac{2\kappa_{{
\rm s},k}}{|x-y|}Y_1(\kappa_{{\rm
s},k}|x-y|)\\
&-\kappa_{{\rm s},k}^2 Y_{0}(\kappa_{{\rm s},k}|x-y|)
-\frac{2\kappa_{{\rm p},k}}{|x-y|}Y_{1}(\kappa_{{\rm p},k}|x-y|)+\kappa_{{\rm
p},k}^2 Y_{0} (\kappa_{{\rm p},k}|x-y|)\Bigr],\\
G_{\rm Re}^{[12]}(x,y,\omega_k)=&-\frac{(x_1-y_1)(x_2-y_2)}{4\omega_k^2|x-y|^2}
\Bigl[ \frac{2\kappa_{{\rm s},k}}{|x-y|}Y_{1}(\kappa_{{\rm s},k}|x-y|)
  -\kappa_{{\rm s},k}^2 Y_{0}(\kappa_{{\rm s},k}|x-y|)\\
  &-\frac{2\kappa_{{\rm
p},k}}{|x-y|}Y_{1}(\kappa_{{\rm p},k}|x-y|)+\kappa_{{\rm p},k}^2 Y_{0}
(\kappa_{{\rm p},k}|x-y|)\Bigr],\\
  G_{\rm Re}^{[21]}(x,y,\omega_k)=&G_{\rm Re}^{[12]}(x,y,\omega_k),\\
G_{\rm
Re}^{[22]}(x,y,\omega_k)=&-\frac{1}{4\mu}Y_{0}(\kappa_{{\rm s},k}|x-y|)+\frac{1}
{4\omega_k^2|x-y|}\Bigl(\kappa_{{\rm
s},k}Y_{1}(\kappa_{{\rm s},k}|x-y|)\\
&-\kappa_{{\rm p},k}Y_1(\kappa_{{\rm p},k}|x-y|)\Bigr)
-\frac{(x_2-y_2)^2}{4\omega_k^2|x-y|^2}\Bigl[\frac{2\kappa_{{\rm
s},k}}{|x-y|}Y_1(\kappa_{{\rm
s},k}|x-y|)\\
&-\kappa_{{\rm s},k}^2 Y_{0}(\kappa_{{\rm s},k}|x-y|)
-\frac{2\kappa_{{\rm p},k}}{|x-y|}Y_{1}(\kappa_{{\rm p},k}|x-y|)+\kappa_{{\rm
p},k}^2 Y_{0} (\kappa_{{\rm p},k}|x-y|)\Bigr],
  \end{split}
\]
and
\[
  {\rm Im}\mathbb{G}(x,y,\omega_k)=\begin{bmatrix}
    G_{\rm Im}^{[11]}(x,y,\omega_k)  &  G_{\rm Im}^{[12]}(x,y,\omega_k)\\[2pt]
    G_{\rm Im}^{[21]}(x,y,\omega_k)  &  G_{\rm Im}^{[22]}(x,y,\omega_k)
  \end{bmatrix},
\]
where
\[
  \begin{split}
    G_{\rm
Im}^{[11]}(x,y,\omega_k)=&\frac{1}{4\mu}J_{0}(\kappa_{{\rm s},k}|x-y|)-\frac{1}{
4\omega_k^2|x-y|}\Bigl(\kappa_{{\rm
s},k}J_{1}(\kappa_{{\rm s},k}|x-y|)\\
&-\kappa_{{\rm p},k}J_1 (\kappa_{{\rm p},k}
|x-y|)\Bigr)+\frac{(x_1-y_1)^2}{4\omega_k^2|x-y|^2}\Bigl[\frac{2\kappa_{{\rm
s},k}}{|x-y|}J_1(\kappa_{{\rm
s},k}|x-y|)\\
&-\kappa_{{\rm s},k}^2 J_{0}(\kappa_{{\rm s},k}|x-y|)-\frac{2\kappa_{{\rm
p},k}}{|x-y|}J_{1}(\kappa_{{\rm p},k}|x-y|)+\kappa_{{\rm p},k}^2 J_{0}
(\kappa_{{\rm p},k}|x-y|)\Bigr],\\
  G_{\rm
Im}^{[12]}(x,y,\omega_k)=&\frac{(x_1-y_1)(x_2-y_2)}{4\omega_k^2|x-y|^2}\Bigl[
\frac{2\kappa_{{\rm s},k}}{|x-y|}J_{1}(\kappa_{{\rm s},k}|x-y|)
  -\kappa_{{\rm s},k}^2 J_{0}(\kappa_{{\rm s},k}|x-y|)\\
  &-\frac{2\kappa_{{\rm
p},k}}{|x-y|}J_{1}(\kappa_{{\rm p},k}|x-y|)+\kappa_{{\rm p},k}^2 J_{0}
(\kappa_{{\rm p},k}|x-y|)\Bigr],\\
  G_{\rm Im}^{[21]}(x,y,\omega_k)=&G_{\rm Im}^{[12]}(x,y,\omega_k),\\
  G_{\rm
Im}^{[22]}(x,y,\omega_k)=&\frac{1}{4\mu}J_{0}(\kappa_{{\rm s},k}|x-y|)-\frac{1}{
4\omega_k^2|x-y|}\Bigl(\kappa_{{\rm
s},k}J_{1}(\kappa_{{\rm s},k}|x-y|)\\
&-\kappa_{{\rm p},k}J_1 (\kappa_{{\rm
p},k}|x-y|)\Bigr)+\frac{(x_2-y_2)^2}{4\omega_k^2|x-y|^2}\Bigl[\frac{2\kappa_{{
\rm s},k}}{|x-y|}J_1(\kappa_{{\rm
s},k}|x-y|)\\
&-\kappa_{{\rm s},k}^2 J_{0}(\kappa_{{\rm s},k}|x-y|)-\frac{2\kappa_{{\rm
p},k}}{|x-y|}J_{1}(\kappa_{{\rm p},k}|x-y|)+\kappa_{{\rm p},k}^2 J_{0}
(\kappa_{{\rm p},k}|x-y|)\Bigr].
\end{split}
\]
Here $J_0$, $Y_0$ and $J_1$,  $Y_1$ are the Bessel function of the first and
second kind with
order zero and order 1, respectively. Clearly, the matrices ${\rm
Re}\mathbb{G}(x,y,\omega_k)$ and ${\rm Im}\mathbb{G}(x,y,\omega_k)$ are
symmetric.

The mild solution \eqref{sd} can be split into the real and imaginary
parts:
\begin{equation}\label{sdr}
{\rm Re}\boldsymbol{u}(x, \omega_k)=\int_D {\rm Re}\mathbb{G}(x, y,
\omega_k)\boldsymbol{g}(y){\rm d}y
+ \int_D {\rm Re}\mathbb{G}(x, y, \omega_k) \boldsymbol{h} (y){\rm d}W_y
\end{equation}
and
\begin{equation}\label{sdi}
{\rm Im}\boldsymbol{u}(x, \omega_k)=\int_D {\rm Im}\mathbb{G}(x, y,
\omega_k)\boldsymbol{g}(y){\rm d}y +
\int_D{\rm Im}\mathbb{G}(x, y, \omega_k) {\bf h} (y){\rm d}W_y.
\end{equation}
Noting
\[
 {\bf E}\Bigl(\int_D{\rm Re}\mathbb{G}(x, y,
\omega_k) \boldsymbol{h} (y){\rm d}W_y\Bigr)=0\quad\text{and}\quad {\bf
E}\Bigl(\int_D{\rm Im}\mathbb{G}(x, y, \omega_k) \boldsymbol{h} (y){\rm
d}W_y\Bigr)=0,
\]
we take the expectation on both sides of \eqref{sdr} and \eqref{sdi} and obtain
real-valued Fredholm integral equations of the first kind to reconstruct
$\boldsymbol{g}$:
\begin{align*}
& {\bf E}({\rm Re}\boldsymbol{u}(x, \omega_k))=\int_D {\rm Re}\mathbb{G}(x, y,
\omega_k) \boldsymbol{g}(y){\rm d}y,\\
& {\bf E}({\rm Im}\boldsymbol{u}(x, \omega_k))=\int_D {\rm Im}\mathbb{G}(x, y,
\omega_k)\boldsymbol{g}(y){\rm d}y,
\end{align*}
which are equivalent to the following equations in component-wise forms:
\begin{align}
  & {\bf E} ( {\rm Re} u_1( x, \omega_k) )=\int_{D} \Big[ G_{\rm Re}^{[11]}( x,
y,\omega_k )g_1(y)+G_{\rm Re}^{[12]}( x, y,\omega_k )g_2(y)\Big]{\rm
d}y,\label{rgr1}\\
  & {\bf E} ( {\rm Re} u_2( x, \omega_k) )=\int_{D} \Big[ G_{\rm Re}^{[21]}( x,
y,\omega_k )g_1(y)+G_{\rm Re}^{[22]}( x, y,\omega_k )g_2(y)\Big]{\rm
d}y,\label{rgr2}\\
  & {\bf E} ( {\rm Im} u_1( x, \omega_k) )=\int_{D} \Big[ G_{\rm Im}^{[11]}( x,
y,\omega_k )g_1(y)+G_{\rm Im}^{[12]}( x, y,\omega_k )g_2(y)\Big]{\rm
d}y,\label{rgi1}\\
  & {\bf E} ( {\rm Im} u_2( x, \omega_k) )=\int_{D} \Big[ G_{\rm Im}^{[21]}( x,
y,\omega_k )g_1(y)+G_{\rm Im}^{[22]}( x, y,\omega_k )g_2(y)\Big]{\rm
d}y.\label{rgi2}
\end{align}

\begin{figure}
\centering
\includegraphics[width=0.45\textwidth]{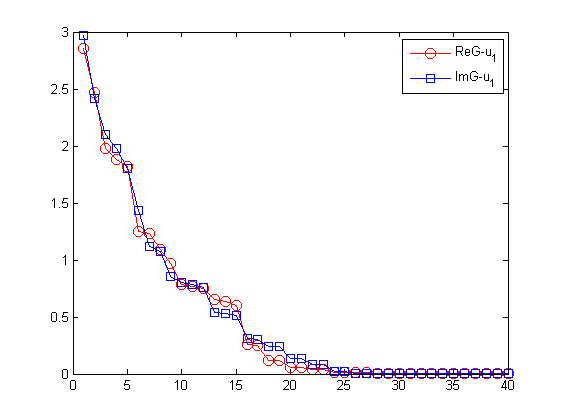}
\includegraphics[width=0.45\textwidth]{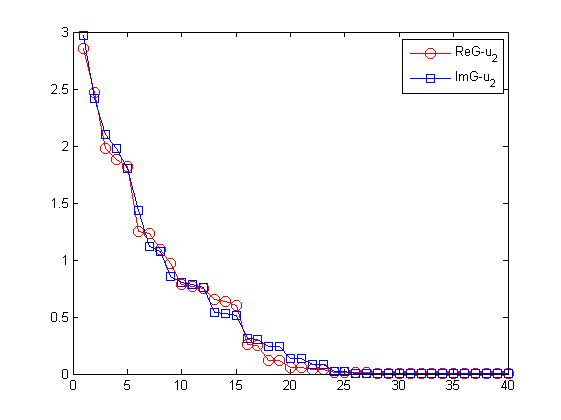}
\caption{Singular values of the Fredholm integral equations for the
reconstruction of $\boldsymbol{g}$: (left) component for $u_1$; (right)
component for $u_2$.}
\label{svdg}
\end{figure}

It is known that Fredholm integral equations of the first kind are ill-posed
due to the rapidly decaying singular values of matrices from the discretized
integral kernels. Appropriate regularization methods are needed to recover the
information about the solutions as stably as possible.
As a representative example, Figure \ref{svdg} plots the singular values of the
matrices for the Fredholm integral equations \eqref{rgr1}--\eqref{rgi2} at
$\omega=1.9\pi$. We can observe similar decaying patterns for the singular
values of \eqref{rgr1} and \eqref{rgi1} for the component $u_1$, and of
\eqref{rgr2} and \eqref{rgi2} for the component $u_2$.

To reconstruct the variance $\boldsymbol{h}^2$,  we use Proposition \ref{propn}
to obtain 
\[
\begin{split}
 &{\bf E}\Bigl(\Bigl|\int_D {\rm Re}\mathbb{G}(x, y, \omega_k) \boldsymbol{h}
(y){\rm d}W_y\Bigr|^2\Bigr)
=\int_D \|{\rm Re}\mathbb{G}(x, y, \omega_k) \boldsymbol{h} (y)\|^{2} {\rm
d}y\\
=&\int_{D}\Big[(G_{\rm Re}^{[11]}(x,y,\omega_k))^2+(G_{\rm
Re}^{[21]}(x,y,\omega_k))^2\Big]h_{1}^{2}(y){\rm d}y\\
&+\int_{D}\Big[(G_{\rm Re}^{[12]}(x,y,\omega_k))^2+(G_{\rm
Re}^{[22]}(x,y,\omega_k))^2\Big]h_{2}^{2}(y){\rm d}y
\end{split}
\]
and
\[
\begin{split}
 &{\bf E}\Bigl(\Bigl|\int_D {\rm Im}\mathbb{G}(x, y, \omega_k) \boldsymbol{h}
(y){\rm d}W_y \Bigr|^2\Bigr)
=\int_D \|{\rm Im}\mathbb{G}(x, y, \omega_k) \boldsymbol{h} (y)\|^{2} {\rm
d}y\\
=&\int_{D}\Big[(G_{\rm Im}^{[11]}(x,y,\omega_k))^2+(G_{\rm
Im}^{[21]}(x,y,\omega_k))^2\Big]h_{1}^{2}(y){\rm d}y\\
&+\int_{D}\Big[(G_{\rm Im}^{[12]}(x,y,\omega_k))^2+(G_{\rm
Im}^{[22]}(x,y,\omega_k))^2\Big]h_{2}^{2}(y){\rm d}y.
\end{split}
\]
Taking the variance on both sides of \eqref{sdr} and \eqref{sdi}, we get
\begin{align*}
 {\bf V}({\rm Re}\boldsymbol{u}(x, \omega_k))=&\int_{D}\Big[(G_{\rm
Re}^{[11]}(x,y,\omega_k))^2+(G_{\rm
Re}^{[21]}(x,y,\omega_k))^2\Big]h_{1}^{2}(y){\rm d}y\\
&+\int_{D}\Big[(G_{\rm Re}^{[12]}(x,y,\omega_k))^2+(G_{\rm
Re}^{[22]}(x,y,\omega_k))^2\Big]h_{2}^{2}(y){\rm d}y,\\
{\bf V}({\rm Im}\boldsymbol{u}(x, \omega_k))=&\int_{D}\Big[(G_{\rm
Im}^{[11]}(x,y,\omega_k))^2+(G_{\rm
Im}^{[21]}(x,y,\omega_k))^2\Big]h_{1}^{2}(y){\rm d}y\\
&+\int_{D}\Big[(G_{\rm Im}^{[12]}(x,y,\omega_k))^2+(G_{\rm
Im}^{[22]}(x,y,\omega_k))^2\Big]h_{2}^{2}(y){\rm d}y,
\end{align*}
which are the Fredholm integral equations of the first kind to reconstruct the
variance. Again, we consider the variance of components $u_1$ and $u_2$:
\begin{align}
  {\bf V}({\rm Re}u_1(x, \omega_k)) &= \int_{D}\Big[\big(G_{\rm
Re}^{[11]}(x,y,\omega_k)\big)^2 h_1^2(y)
  +\big(G_{\rm Re}^{[12]}(x,y,\omega_k)\big)^2h_2^2(y)\Big]{\rm
d}y,\label{rhr1}\\
  {\bf V}({\rm Re}u_2(x, \omega_k)) &= \int_{D}\Big[\big(G_{\rm
Re}^{[21]}(x,y,\omega_k)\big)^2 h_1^2(y)
  +\big(G_{\rm Re}^{[22]}(x,y,\omega_k)\big)^2h_2^2(y)\Big]{\rm
d}y,\label{rhr2}\\
  {\bf V}({\rm Im}u_1(x, \omega_k)) &= \int_{D}\Big[\big(G_{\rm
Im}^{[11]}(x,y,\omega_k)\big)^2 h_1^2(y)
  +\big(G_{\rm Im}^{[12]}(x,y,\omega_k)\big)^2 h_2^2(y)\Big]{\rm
d}y,\label{rhi1}\\
  {\bf V}({\rm Im}u_2(x, \omega_k)) &= \int_{D}\Big[\big(G_{\rm
Im}^{[21]}(x,y,\omega_k)\big)^2 h_1^2(y)
  +\big(G_{\rm Im}^{[22]}(x,y,\omega_k)\big)^2h_2^2(y)\Big]{\rm d}y. 
\label{rhi2}
\end{align}

To investigate ill-posedness of the above four equations, we plot their
singular values in Figure \ref{svdh}. It can be seen that
\eqref{rhr1}, \eqref{rhi1} and \eqref{rhr2}, \eqref{rhi2}  show almost
identical distributions of the singular values for components $u_1$ and $u_2$, 
respectively. The singular values decay exponentially to zeros and there is a
big gap between the few leading singular values and the
rests. Hence it is severely ill-posed to use directly either \eqref{rhr1} or
\eqref{rhi1} and \eqref{rhr2} or \eqref{rhi2} to reconstruct $h_1^2$ and
$h_2^2$. Subtracting \eqref{rhi1} from \eqref{rhr1} and \eqref{rhi2} from
\eqref{rhr2}, we obtain the improved equations to reconstruct
$h_1^2$ and $h_2^2$:
\begin{align}
 \label{rh2d}{\bf V}({\rm Re}u_1(x, \omega_k))-{\bf V}({\rm
Im}u_1(x, \omega_k))=\int_{D}\Big[\big(G_{\rm
Re}^{[11]}(x,y,\omega_k)\big)^2-\big(G_{\rm
Im}^{[11]}(x,y,\omega_k)\big)^2\Big]h_1^2(y){\rm d}y\nonumber\\
  +\int_{D}\Big[\big(G_{\rm Re}^{[12]}(x,y,\omega_k)\big)^2-\big(G_{\rm
Im}^{[12]}(x,y,\omega_k)\big)^2\Big]h_2^2(y){\rm d}y,\\
\label{rh3d} {\bf V}({\rm Re}u_2(x, \omega_k))-{\bf V}({\rm
Im}u_2(x, \omega_k))=\int_{D}\Big[\big(G_{\rm
Re}^{[21]}(x,y,\omega_k)\big)^2-\big(G_{\rm
Im}^{[21]}(x,y,\omega_k)\big)^2\Big]h_1^2(y){\rm d}y\nonumber\\
  +\int_{D}\Big[\big(G_{\rm Re}^{[22]}(x,y,\omega_k)\big)^2-\big(G_{\rm
Im}^{[22]}(x,y,\omega_k)\big)^2\Big]h_2^2(y){\rm d}y.
\end{align}
In fact, it is clear to note in Figure \ref{svdh} that the singular values of
\eqref{rh2d} and \eqref{rh3d} display better behavior that those
of \eqref{rhr1}, \eqref{rhi1} and \eqref{rhr2}, \eqref{rhi2}. The singular
values decay more slowly and distribute more uniformly. Numerically,
\eqref{rh2d} and \eqref{rh3d} do give much better reconstructions. We will only
show the results for \eqref{rh2d} and \eqref{rh3d} in the numerical experiments.

\begin{figure}
\centering
\includegraphics[width=0.45\textwidth]{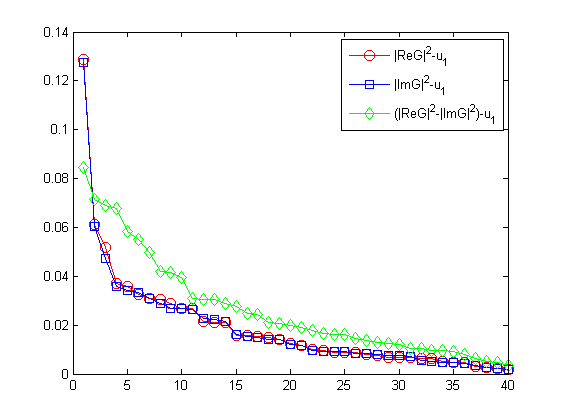}
\includegraphics[width=0.45\textwidth]{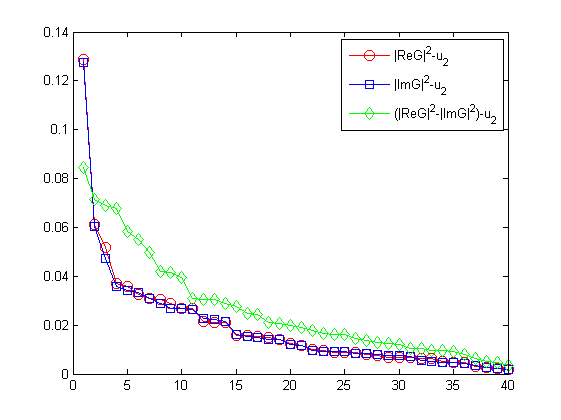}
\caption{Singular values of the Fredholm integral equations for the
reconstruction of $ \boldsymbol{h} ^2$: (left) component of $u_1$; (right)
component of $u_2$.}
\label{svdh}
\end{figure}

\subsection{Numerical method}

In this section, we present a regularized Kaczmarz method to solve the
ill-posed integral equations. The classical Kaczmarz method is an iterative
method for solving linear systems of algebraic equations \cite{N-86}.

Consider the following operator equations
\begin{equation}\label{ae}
 A_k q=p_k,\quad k=1, \dots, K,
\end{equation}
where the index $k$ is for different frequency, $q$ represents the unknown
$g_1, g_2$ or $h_1^2, h_2^2$, and $p_k$ is the given data. Given an arbitrary
initial guess $q^0$, the classical Kaczmarz method for solving \eqref{ae} reads:
For $l=0, 1, \dots, L$,
\begin{equation}\label{ckm}
 \begin{cases}
  q_0 = q^l,\\
  q_k=q_{k-1}+A^*_k (A_k A^*_k)^{-1}(p_k-A_k q_{k-1}),\quad k=1, \dots, K,\\
  q^{l+1}=q_K,
 \end{cases}
\end{equation}
where $A_k^*$ is the adjoint operator of $A_k$. In \eqref{ckm}, there are two
loops: the outer loop is carried for iterative index $l$ and the inner loop is
done for the different frequency $\omega_k$. In practice, the operator $A_k
A^*_k$ may not be invertible or is bad conditioned even if it is invertible. A
regularization technique is needed.

We present a regularized Kaczmarz method: Given an arbitrary
initial guess $q^0$,
\begin{equation}\label{rkm}
 \begin{cases}
  q_0 = q^l,\\
  q_k=q_{k-1}+A^*_k (\gamma I + A_k A^*_k)^{-1}(p_k-A_k q_{k-1}),\quad k=1,
\dots, K,\\
  q^{l+1}=q_m,
 \end{cases}
\end{equation}
for $l=0, 1, \dots, L$, where $\gamma>0$ is the regularization parameter and $I$
is the identity operator. Although there are two loops in \eqref{rkm}, the
operator $\gamma I+A_k A^*_k$ leads to a small scale linear system of equations
with the size equal to the number of measurements. Moreover, they essentially
need to be solved only $K$ times by a direct solver such as the LU decomposition
since $A_k$ keep the same during the outer loop.

\section{Numerical experiments}

In this section, we present a numerical example to demonstrate the validity
and effectiveness of the proposed method. The scattering data is obtained by the
numerical solution of the stochastic Navier equation instead of the numerical
integration of the Fredholm integral equations in order to avoid the so-called
inverse crime. Although the stochastic Navier equation may be efficiently solved
by using the Wiener Chaos expansions to obtain statistical moments such as the
mean and variance \cite{BAZC-MMS10}, we choose the Monte Carlo method to
simulate the actual process of measuring data. In each realization, the
stochastic Navier equation is solved by using the finite element method with the
perfectly matched layer (PML) technique. After all the realizations are done, we
take an average of the solutions and use it as an approximated scattering data
to either the mean or the variance. It is clear to note that the data is more
accurate as more number of realizations is taken.

\begin{figure}
\centering
\includegraphics[width=0.45\textwidth]{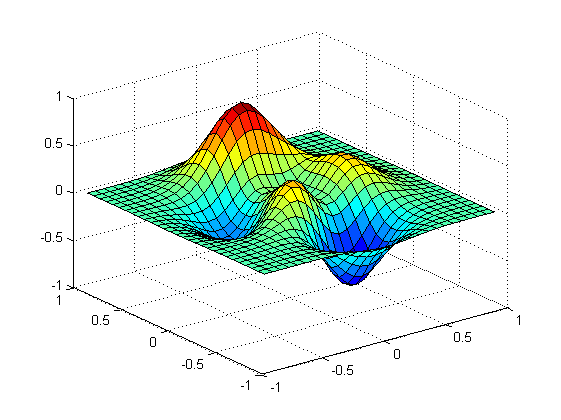}
\includegraphics[width=0.45\textwidth]{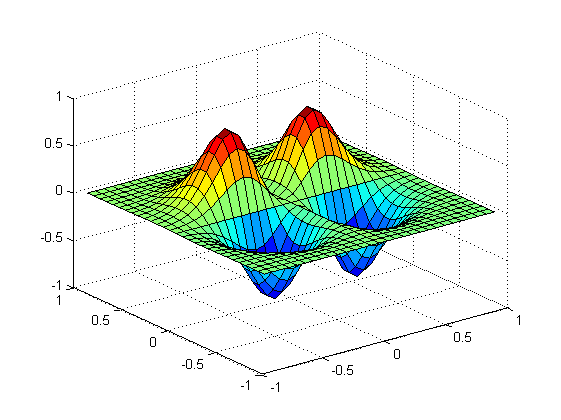}
\includegraphics[width=0.45\textwidth]{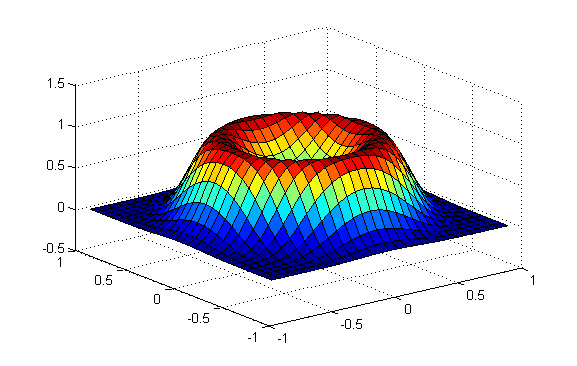}
\includegraphics[width=0.45\textwidth]{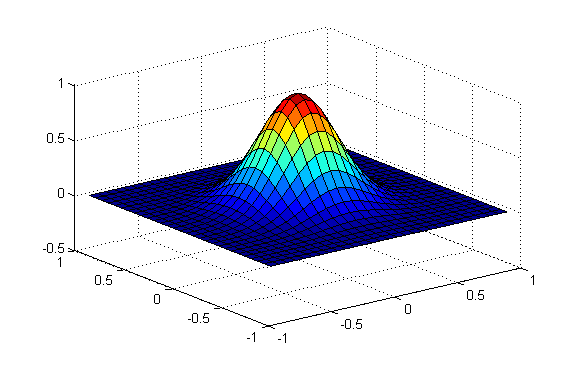}
\caption{The exact source: (top) surface plot of the exact mean
$g_1$ and $g_2$; (below) surface plot of the exact variance $h_1^2$ and
$h_2^2$.}
\label{gtht2d}
\end{figure}

Let
\[
g_{1}(x_1, x_2)=0.3(1-x_1)^2
e^{-x_1^2-(x_2+1)^2}-(0.2x_1-x_1^3-x_2^5)e^{-x_1^2-x_2^2} -0.03e^{
-(x_1+1)^2-x_2^2}
\]
and
\[
g_{2}(x_1, x_2)=5x_1^2 x_2 e^{-x_1^2-x_2^2}.
\]
We reconstruct the mean $\boldsymbol{g}$ given by
\[
\boldsymbol{g}(x_1, x_2)=(g_1(3x_1, 3x_2), g_{2}(3x_1,
3x_2))^\top
\]
inside the domain $D=[-1,\,1]\times[-1,\,1]$. Let
\[
  h_1 (x_1, x_2) = 0.6e^{-8(r^3-0.75r^2)}
\]
and
\[
  h_2 (x_1, x_2)= e^{-r^2},\quad r=(x_1^2+x_2^2)^{1/2}.
\]
We reconstruct the variance $ \boldsymbol{h}^2 $ given by
\[
\boldsymbol{h}^2 (x_1, x_2) = \begin{bmatrix} h^2_1(x_1, x_2) & 0\\ 0&
h^2_2(3x_1, 3x_2) \end{bmatrix}
\]
inside the domain $D=[-1,\,1]\times[-1,\,1]$. See Figure \ref{gtht2d} for the
surface plot of the exact $g_1, g_2$ (top) and $h_1^2, h_2^2$ (below). The
Lam\'{e} constants are $\mu=1.0$ and $\lambda=2.0$. The computational domain is
set to be $[-3,\,3]\times[-3,\,3]$ with the PML thickness $0.5$. After the
direct problem is solved and the value of $\boldsymbol{u}$ is obtained at the
grid points, the linear interpolation is used to generate the synthetic data at
40 uniformly distributed points on the circle with radius 2, i.e.,
$x_1=2\cos\theta_i, x_2=2\sin\theta_i, \theta_i=i\pi/20, i=0, 1, \dots, 39.$
Sixteen equally spaced frequencies from $0.5\pi$ to $7.5\pi$ are used in the
reconstruction of $g_1, g_2$, while twenty equally spaced frequencies from
$0.5\pi$ to $2.5\pi$ are used in the reconstruction of $h_1^2, h_2^2$. The
regularization parameter is $\gamma=1.0\times 10^{-7}$ and the total number of
the outer loop for the Kaczmarz method is $L=5$. Figure \ref{gnhn2d} shows the
reconstructed mean $g_1, g_2$ (top) and the reconstructed variance $h_1^2,
h_2^2$ (below) corresponding to the number of realization $10^4$.

\begin{figure}
\centering
\includegraphics[width=0.45\textwidth]{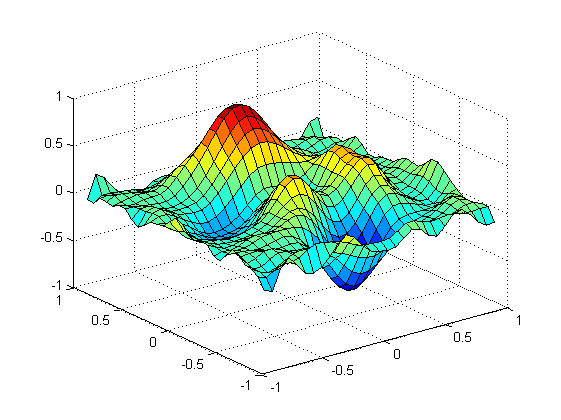}
\includegraphics[width=0.45\textwidth]{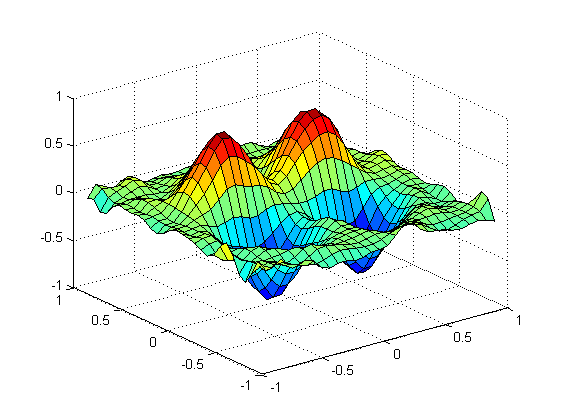}
\includegraphics[width=0.45\textwidth]{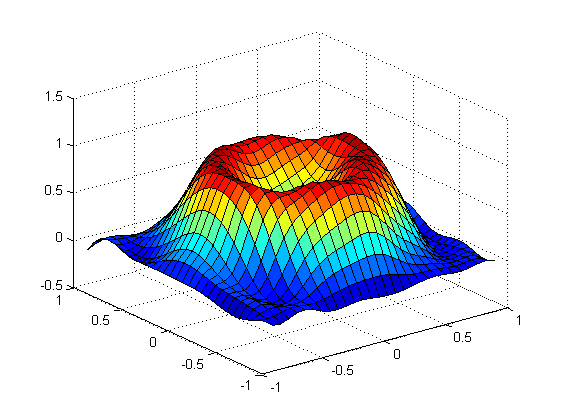}
\includegraphics[width=0.45\textwidth]{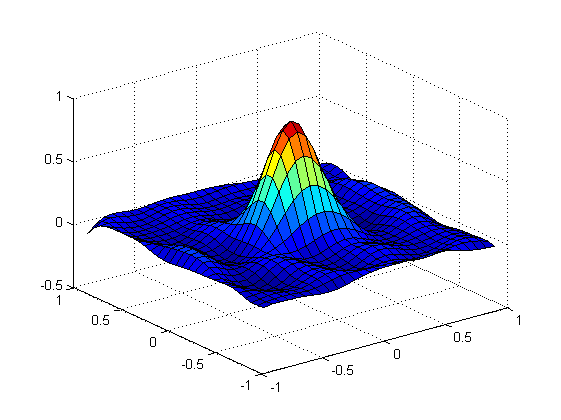}
\caption{The reconstructed source: (top) surface plot of the reconstructed
mean $g_1, g_2$; (below) surface plot of the reconstructed variance $h_1^2,
h_2^2$.}
\label{gnhn2d}
\end{figure}

\section{Conclusion}

We have studied the direct and inverse random source scattering problems for the
stochastic Navier equation where the source is driven by an additive white
noise. Under a suitable regularity assumption of the source functions
$\boldsymbol{g}$ and $ \boldsymbol{h} $, the direct scattering problem is shown
constructively to have a unique mild solution which was given explicitly as an
integral equation. Based on the explicit solution, Fredholm integral equations
are deduced for the inverse scattering problem to reconstruct the mean and the
variance of the random source. We have presented the regularized Kaczmarz method
to solve the ill-posed integral equations by using multiple frequency data.
A numerical example is presented to demonstrate the validity and effectiveness
of the proposed method. We are currently investigating the inverse random source
scattering problem in an inhomogeneous elastic medium where the explicit Green
tensor function is no longer available. Although this paper concerns the
inverse random source scattering problem for the Navier equation, we believe
that the proposed framework and methodology can be directly applied to solve
many other inverse random source problems and even more general stochastic
inverse problems. For instance, it is interesting to study inverse random source
problems for the stochastic Poisson, heat, or electromagnetic wave equation.
Obviously, it is more challenging to consider the inverse random medium
scattering problem where the medium should be modeled as a random function. We
hope to be able to report the progress on these problems in the future.

\appendix

\section{Brownian sheet}

Let us first briefly introduce the one-dimensional Brownian sheet, which is also
called one-dimensional $d$-parameter Brownian motion, on $(\mathbb{R}^d_+,\,
\mathcal{B}(\mathbb{R}^d_+),\,\mu)$, where $d\in\mathbb{N},
\mathbb{R}_+^d=\{x=(x_1, \dots, x_d)^\top\in\mathbb{R}^d:
x_j\geq 0, j=1,\dots, d\}$, $\mathcal{B}(\mathbb{R}^d_+)$ is the Borel
$\sigma$-algebra of $\mathbb{R}^d_+$, and $\mu$ is the Lebesgue measure. More
details can be found in \cite{W-86}. Let
$(0, x]=(0, x_1]\times\cdots\times(0, x_d]$ for $x\in\mathbb{R}^d_+$. 

\begin{definition}\label{bs}
The one-dimensional Brownian sheet on $\mathbb{R}^d_+$ is the process $\{W_x:
x\in\mathbb{R}^d_+\}$ defined by $W_x=W\{(0, x]\}$, where $W$ is a random set
function such that
\begin{enumerate}

\item $\forall A\in\mathcal{B}(\mathbb{R}^d_+)$, $W(A)$ is a
Gaussian random variable with mean 0 and variance $\mu(A)$, i.e., $W(A)\sim
\mathcal{N}(0, \mu(A))$;

\item $\forall A, B\in\mathcal{B}(\mathbb{R}^d_+)$, if $A\cap B=\emptyset$, then
$W(A)$ and $W(B)$ are independent and $W(A\cup B)=W(A)+W(B)$.

\end{enumerate}
\end{definition}

It can be verified from Definition \ref{bs} that
\[
{\bf E}(W(A)W(B))=\mu(A\cap B),\quad\forall A, B\in\mathcal{B}(\mathbb{R}^d_+),
\]
which gives the covariance function of the Brownian sheet:
\[
 {\bf E}(W_x W_y)=x\wedge y:=(x_1\wedge y_1)\cdots (x_d\wedge y_d)
\]
for any $x=(x_1, \dots, x_d)^\top\in\mathbb{R}^d_+$ and $y=(y_1, \dots,
y_d)^\top\in\mathbb{R}^d_+$, where $x_j\wedge y_j=\min\{x_j,\,y_j\}$.

The Brownian sheet can be generalized to the space $\mathbb{R}^d$ by introducing
$2^d$ independent Brownian sheets defined on $\mathbb{R}^d_+$. Define a
multi-index $t=(t_1, \dots, t_d)^\top$ with $t_j=\{1,\,-1\}$ for
$j=1, \dots, d$. Introduce $2^d$ independent Brownian sheets $\{W^t\}$
defined on $\mathbb{R}^d_+$. For any $x=(x_1, \dots, x_d)^\top\in\mathbb{R}^d$,
define the Brownian sheet
\[
W_x:=W^{t(x)}_{\breve{x}},
\]
where $\breve{x}=(|x_1|, \dots, |x_d|)^\top$ and $t(x)=({\rm sgn}(x_1),
\dots, {\rm sgn}(x_d))^\top$. The sign function ${\rm sgn}(x_j)=1$ if $x_j\geq
0$, otherwise ${\rm sgn}(x_j)=-1$.

In two or more parameters, the white noise can be thought of as the derivative
of the Brownian sheet. In fact, the Brownian sheet $W_x$ is
nowhere-differentiable
in the ordinary sense, but its derivatives will exist in the sense of Schwartz
distributions. Define
\[
\dot{W}_x=\frac{\partial^d W_x}{\partial x_1\cdots\partial x_d}.
\]
If $\phi(x)$ is a deterministic square-integrable complex-valued test function
with a compact support in $\mathbb{R}^d$, then $\dot{W}_x$ is the
distribution
\[
 \dot{W}_x(\phi)=(-1)^d \int_{\mathbb{R}^d} W_x \frac{\partial^d
\phi(x)}{\partial x_1\cdots\partial x_d}{\rm d}x.
\]
We may define the stochastic integral
\begin{equation}\label{si}
 \int_{\mathbb{R}^d}\phi(x){\rm d}W_x=(-1)^d \int_{\mathbb{R}^d} W_x
\frac{\partial^d \phi(x)}{\partial x_1\cdots\partial x_d}{\rm d}x,
\end{equation}
which satisfies the following properties (cf. \cite[Proposition A.2]{BCL}). 

\begin{lemma}\label{prop}
Let $\phi(x)$ be a test function with a compact support in $\mathbb{R}^d$. We
have
\[
 {\bf E} \Bigl(\int_{\mathbb{R}^d}\phi(x){\rm d}W_x\Bigr)=0,\quad
 {\bf E} \Bigl(\bigl|\int_{\mathbb{R}^d} \phi(x){\rm
d}W_x\bigr|^2 \Bigr)=\int_{\mathbb{R}^d}|\phi(x)|^2{\rm d}x.
\]
\end{lemma}

The stochastic integral \eqref{si} can be extended to define the
multi-dimensional stochastic integrals. Let $W(x)=(W_1(x), \dots,
W_{n}(x))^\top$ be an $n$-dimensional Brownian sheet, where $W_i(x)$ and
$W_j(x)$ are two one-dimensional independent Brownian sheets for $i\neq j$. If
$\phi(x)$ is a deterministic square-integrable complex-valued $m\times n$
matrix-valued test function with each component compactly supported in
$\mathbb{R}^d$, i.e.,
\[
\phi(x)=\begin{bmatrix}
  \phi_{11}(x) & \cdots & \phi_{1n}(x)\\
  \vdots &   &\vdots\\
  \phi_{m1}(x) &\cdots &\phi_{mn}(x)
\end{bmatrix}.
\]
Using the matrix notation, we may define the multi-dimensional stochastic
integral
\begin{equation}\label{msi}
\int_{{\mathbb R}^{d}}\phi(x){\rm d}W(x)=\int_{{\mathbb R}^{d}}
\begin{bmatrix}
  \phi_{11}(x) & \cdots & \phi_{1n}(x)\\
  \vdots &   &\vdots\\
  \phi_{m1}(x) &\cdots &\phi_{mn}(x)
\end{bmatrix}
\begin{bmatrix}
 {\rm d}W_{1}(x)\\
  \vdots\\
 {\rm d}W_{n}(x)
\end{bmatrix},
\end{equation}
which is an $m\times 1$ matrix and its $j$-th component is the 
sum of 1-dimensional stochastic integrals
\[
\sum_{k=1}^{n}\int_{{\mathbb R}^{d}}\phi_{jk}(x){\rm d}W_{k}(x).
\]

We have the similar properties for the multi-dimensional stochastic integral.

\begin{proposition}\label{propn}
Let $W(x)=(W_1(x), \dots, W_n(x))^\top$ be an $n$-dimensional Brownian sheet and
$\phi(x)=(\phi_{ij}(x))_{m\times n}$ be an $m\times n$ matrix-valued function
with each component $\phi_{ij}(x)$ compactly supported in $\mathbb{R}^d$. We
have 
\[
 {\bf E} \Bigl(\int_{\mathbb{R}^d}\phi(x){\rm d}W_x\Bigr)=0,\quad
 {\bf E} \Bigl(\bigl|\int_{\mathbb{R}^d} \phi(x){\rm
d}W_x\bigr|^2 \Bigr)=\int_{\mathbb{R}^d}\|\phi(x)\|^2{\rm d}x,
\]
where $\|\cdot\|$ is the Frobenius norm. 
\end{proposition}

\begin{proof}
 It follows from \eqref{msi} and Lemma \ref{prop} that
 \[
  {\bf E}\Bigl(\int_{\mathbb{R}^d}\phi(x){\rm d}W_x\Bigr)=\begin{bmatrix}
  \sum_{k=1}^{n}{\bf E}\Bigl(\int_{{\mathbb R}^{d}}\phi_{jk}(x){\rm
d}W_{k}(x)\Bigr)\\
\vdots\\
\sum_{k=1}^{n}{\bf E}\Bigl(\int_{{\mathbb R}^{d}}\phi_{jk}(x){\rm
d}W_{k}(x)\Bigr)
\end{bmatrix}=0.
 \]

Using \eqref{msi} and Lemma \ref{prop} again, we have
\begin{align*}
 {\bf E}\Bigl(\bigl|\int_{\mathbb{R}^d} \phi(x){\rm
d}W_x\bigr|^2\Bigr)
&=\sum_{j=1}^{m}{\bf E}\Bigl(\bigl|\sum_{k=1}^{n}\int_{\mathbb{R}^d}
\phi_{jk}(x){\rm
d}W_k(x)\bigr|^2\Bigr)\\
&=\sum_{j=1}^{m}\sum_{k=1}^{n}\int_{\mathbb{R}^d} |\phi_{jk}(x)|^2{\rm
d}x\\
&=\int_{\mathbb{R}^d} \|\phi(x)\|^2{\rm
d}x,
\end{align*}
which completes the proof.
\end{proof}


\begin{thebibliography}{00}

\bibitem{AM-IP06}
\textsc{R. Albanese and P. Monk}, The inverse source problem for Maxwell's
equations,
Inverse Problems, 22 (2006), 1023--1035.

\bibitem{ABF-SJAM02}
\textsc{H. Ammari, G. Bao, and J. Fleming}, An inverse source problem for
Maxwell's equations in magnetoencephalography, SIAM J. Appl. Math., 62 (2002),
1369--1382.

\bibitem{BN-IP11}
\textsc{A. Badia and T. Nara}, An inverse source problem for Helmholtz's
equation from the Cauchy data with a single wave number, Inverse Problems, 27
(2011), 105001.

\bibitem{BAZC-MMS10}
\textsc{M. Badieirostami, A. Adibi, H.-M. Zhou, and S.-N. Chow}, Wiener chaos
expansion and simulation of electromagnetic wave propagation excited by a
spatially incoherent source, Multiscale Model. Simul., 8 (2010), pp. 591--604.

\bibitem{BCL}
\textsc{G. Bao, C. Chen and P. Li}, Inverse random source scattering problems in
several dimensions, preprint.

\bibitem{BCLZ-MC14}
\textsc{G. Bao, S.-N. Chow, P. Li, and H.-M. Zhou}, An inverse random source
problem for the Helmholtz equation, Math. Comp., 83 (2014), 215--233.

\bibitem{BLLT-IP15}
\textsc{G. Bao, P. Li, J. Lin, and F. Triki}, Inverse scattering problems with
multi-frequencies, Inverse Problems, 31 (2015), 093001.

\bibitem{BLT-JDE10}
\textsc{G. Bao, J. Lin, and F. Triki}, A multi-frequency inverse source
problem, J. Differential Equations, 249 (2010), 3443--3465.

\bibitem{BLRX-SJNA15}
\textsc{G. Bao, S. Lu, W. Rundell, and B. Xu}, A recursive algorithm for
multifrequency acoustic inverse source problems, SIAM J. Numer. Anal., 53
(2015), 1608--1628.

\bibitem{BX-IP13}
\textsc{G. Bao and X. Xu}, An inverse random source problem in quantifying the
elastic modulus of nano-materials, Inverse Problems, 29 (2013), 015006.

\bibitem{BC-IP05}
\textsc{M. Bonnet and A. Constantinescu}, Inverse problems in elasticity,
Inverse Problems, 21 (2005) 1--50, 

\bibitem{CZZ-JCM08}
\textsc{Y.-Z. Cao, R. Zhang, and K. Zhang}, Finite element method and
discontinuous Galerkin method for stochastic Helmholtz equation in two- and
three-dimensions, J. Comput. Math., 26 (2008), 701--715.

\bibitem{CK-98}
\textsc{D. Colton and R. Kress}, Inverse Acoustic and Electromagnetic Scattering
Theory, Berlin: Springer, 1998.

\bibitem{D-JMP79}
\textsc{A. Devaney}, The inverse problem for random sources, J. Math. Phys.,
20 (1979), 1687--1691.

\bibitem{DML-SJAM07}
\textsc{A. Devaney, E. Marengo, and M. Li}, Inverse source problem in
nonhomogeneous background media, SIAM J. Appl. Math., 67 (2007), 1353--1378.

\bibitem{DS-IEEE82}
\textsc{A. Devaney and G. Sherman}, Nonuniqueness in inverse source
and scattering problems, IEEE Trans. Antennas Propag., 30 (1982), 1034--1037.

\bibitem{EV-IP09}
\textsc{M. Eller and N. Valdivia}, Acoustic source identification using
multiple frequency information, Inverse Problems, 25 (2009), 115005.

\bibitem{E-13}
\textsc{L. Evans}, An Introduction to Stochastic Differential Equations, AMS,
2013.

\bibitem{FKM-IP04}
\textsc{A. Fokas, Y. Kurylev, and V. Marinakis}, The unique determination of
neuronal currents in the brain via magnetoencephalogrphy, Inverse Problems, 20
(2004), 1067--1082.

\bibitem{F-06}
\textsc{A. Friedman}, Stochastic Differential Equations and Applications, New
York: Academic Press, 2006.

\bibitem{HKP-IP05}
\textsc{K.-H. Hauer, L. K\"{u}hn, and R. Potthast}, On uniqueness and
non-uniqueness for current reconstruction from magnetic fields, Inverse
Problems, 21 (2005), 955--967.

\bibitem{H-09}
\textsc{M. Hairer}, Introduction to Stochastic PDEs, lecture notes, 2009.

\bibitem{I-89}
\textsc{V. Isakov}, Inverse Source Problems, AMS, Providence, RI, 1989.

\bibitem{KE-05}
\textsc{J. Kaipio and E. Somersalo}, Statistical and Computational Inverse
Problems, Springer-Varlag, New York, 2005.

\bibitem{K-96}
\textsc{R. Kress}, Inverse elastic scattering from a crack, Inverse Problems, 12
(1996), 667--684.

\bibitem{LL-86}
\textsc{L. D. Landau and E. M. Lifshitz}, Theory of Elasticity, Oxford, UK:
Pergamon, 1986. 

\bibitem{L-IP11}
\textsc{P. Li}, An inverse random source scattering problem in inhomogeneous
media, Inverse Problems, 27 (2011), 035004.

\bibitem{LWZ-IP15}
\textsc{P. Li, Y. Wang, and Y. Zhao}, Inverse elastic surface scattering with
near-field data, Inverse Problems, 31 (2015), 035009. 

\bibitem{MD-IEEE99}
\textsc{E. Marengo and A. Devaney}, The inverse source problem of
electromagnetics: linear inversion formulation and minimum energy solution,
IEEE Trans. Antennas Propag., 47 (1999), 410--412.

\bibitem{N-86}
\textsc{F. Natterer}, The Mathematics of Computerized Tomography, Teubner,
Stuttgart, 1986.

\bibitem{W-86}
\textsc{J. Walsh}, An Introduction to Stochastic Partial Differential Equations,
Springer, 1986.

\bibitem{ZG-IP15}
\textsc{D. Zhang and Y. Guo}, Fourier method for solving the multi-frequency
inverse acoustic source problem for the Helmholtz equation, Inverse Problems, 31
(2015), 035007.

\end{thebibliography}
\end{document}